\documentclass[10pt]{amsart}
\usepackage{amsmath,amssymb,mathrsfs,latexsym,amsthm,xcolor,array}
\definecolor{gray}{rgb}{0.4,0.4,0.4}
\definecolor{white}{rgb}{1,1,1}
\definecolor{red}{rgb}{0.8,0,0}

\title{Index Sets of Universal Codes}

\keywords{}
\subjclass[2010]{Primary 03D30, 03D80}

\author{Achilles A. Beros}
\author{Konstantinos A. Beros}

\newtheorem{Theorem}{Theorem}[section]
\newtheorem{Proposition}[Theorem]{Proposition}

\newtheorem{Corollary}[Theorem]{Corollary}

\theoremstyle{definition}
\newtheorem{Definition}[Theorem]{Definition}

\newcommand{\bbn}{\mathbb N}

\newcommand{\indset}[1]{\textsc{#1}}

\newcommand{\baire}{\omega^{\omega}}
\newcommand{\bairestr}{\omega^{<\omega}}

\newcommand{\upto} {{\upharpoonright}}

\begin{document}

\begin{abstract}
We examine sets of codes such that certain properties are invariant under the choice of oracle from a range of possible oracles and establish a connection between such codes and Medvedev reductions.   In examing the complexity of such sets of \emph{universal codes}, we prove completeness results at various levels of the arithmetic hierarchy as well as two general theorems for obtaining $\Pi_1^1$-completeness for sets of universal codes.  Among other corollaries, we show that the set of codes for Medvedev reductions of bi-immune sets to DNC functions is $\Pi_1^1$-complete.
\end{abstract}

\maketitle

\section{Introduction}

Throughout, we will use the notation $\{ e \}^A (x)$ for the result of applying the Turing machine coded by $e$, with oracle $A$, to input $x$.  If the attempt to compute $\{ e \}^A (x)$ halts and gives output $y \in \mathbb N$, we write
\[
\{ e \}^A (x) \downarrow = y.
\]
Otherwise, we write $\{ e \}^A (x) \uparrow$ to indicate that the computation of $\{ e \}^A (x)$ never terminates.  Note the the oracle $A$ may be either a subset of $\bbn$ or a (possibly partial) function on $\bbn$.  In the event that $A$ is a partial function, we require $\{ e \}^A (x)$ to diverge if an oracle query is made for an input not in the domain of $A$.

We let $\{ e \}^A_s (x)$ denote the result of allowing the computation of $\{ e \}^A (x)$ to run for $s$ computation stages.  Finally, $\mathrm{use} (\{ e \}^A_s (x))$ will denote the largest oracle query made by the computations $\{ e \}^A_{s'}$, for $s' \leq s$.

Recall that, for sets $A,B \subseteq\mathbb N$, one says $A$ is {\em Turing reducible} to $B$ (denoted $A \leq_{\mathrm T} B$) if, and only if, there is an $e \in \mathbb N$ such that the computation of $\{ e \}^A (x)$ terminates with output 0 or 1, for each $x \in \mathbb N$, and
\[
(\forall x )(x \in A \iff \{ e \}^B(x) \downarrow = 1).
\]
In this case, we write $\{ e \}^B = \chi_A$ to indicate that the map
\[
x \mapsto \{ e \}^B (x)
\]
is the characteristic function of $A$.  Sets $A, B \subseteq \mathbb N$ are said to be {\em Turing equivalent} (written $A \equiv_{\mathrm T} B$) if, and only if, $A \leq_{\mathrm T} B$ and $B \leq_{\mathrm T} A$.  The {\em Turing degree} of $A \subseteq \mathbb N$ is the family $\{ B \subseteq \mathbb N : B \equiv_{\mathrm T} A\}$ of subsets of $\mathbb N$.

The relation of Turing reducibility gives a natural pre-order on the family of subsets of $\mathbb N$.  With this in mind, it is desirable to have a corresponding pre-order on the subsets of the Turing degrees themselves.  There are several natural ways of obtaining such a pre-order.  Extensive study has been done of two such pre-orders: {\em Muchnik reducibility} and {\em Medvedev reducibility}.

\begin{Definition}\cite{medvedev}\cite{muchnik}
Let $\mathcal A$ and $\mathcal B$ be sets of degrees.
\begin{enumerate}
\item $\mathcal A$ is {\em Muchnik reducible} (or {\em weakly reducible}) to $\mathcal B$ (written $\mathcal A \leq_w \mathcal B$) if 
\[
(\forall B \in \mathcal B)(\exists A \in \mathcal A)(\exists e \in \mathbb N)\Big( \{ e \}^B = \chi_A \Big)
\]
\item $\mathcal A$ is {\em Medvedev reducible} (or {\em strongly reducible}) to $\mathcal B$ (written $\mathcal A \leq_s \mathcal B$) if 
\[
(\exists e \in \mathbb N)(\forall B \in \mathcal B)(\exists A \in \mathcal A)\Big( \{ e \}^B = \chi_A\Big)
\]
\end{enumerate}
The oracle Turing machine determined by $e$ as above is called a Muchnik (resp., Medvedev) reduction of $\mathcal A$ to $\mathcal B$.
\end{Definition}

In the case that $\mathcal A \leq_s \mathcal B$, with $e$ such that
\[
(\forall B \in \mathcal B)(\exists A \in \mathcal A)\Big( \{ e \}^B = \chi_A\Big),
\]
we call $e$ a {\em code} witnessing the Medvedev reduction of $\mathcal A$ to $\mathcal B$.  In the present work, we study the arithmetic complexity of sets of codes of Medvedev reductions associated to various classes $\mathcal A$ and $\mathcal B$.  We make the following definition.

\begin{Definition}\label{D1}
Let $\mathcal A$ and $\mathcal B$ be families of subsets of $\bbn$.  We say that $e \in \bbn$ is a \emph{$\mathcal B$-universal $\mathcal A$-code} if
\[
(\forall B \in \mathcal B) (\exists A \in \mathcal A) \big( \{e\}^B = \chi_A \big).
\]
We let $\mathcal B [\mathcal A]$ denote the set of all $\mathcal B$-universal $\mathcal A$-codes.
\end{Definition}

By definition, a $\mathcal B$-universal $\mathcal A$-code determines a Medvedev reduction of $\mathcal A$ to $\mathcal B$ and $\mathcal B [ \mathcal A]$ corrresponds to the set of Medvedev reductions of $\mathcal A$ to $\mathcal B$.

As a more general case of this definition, we suppose that $\mathcal A$ and $\mathcal B$ are classes of partial functions instead of oracles.  Consider the following:

\begin{Definition}\label{D2}
Suppose that $\mathcal F, \mathcal G$ are families of (possibly partial) functions on $\bbn$.  We say that $e \in \bbn$ is a {\em $\mathcal G$-universal $\mathcal F$-code} if, and only if, the function
\[
x \mapsto \{ e \}^g (x)
\]
is in $\mathcal F$, for each $g \in \mathcal G$.  As before, we let $\mathcal G [ \mathcal F]$ denote the set of $\mathcal G$-universal $\mathcal F$-codes.
\end{Definition}

Definition~\ref{D1} may be regarded as a special case of Definition~\ref{D2} by letting $\mathcal F$ (in Definition~\ref{D2}) be the class of characteristic functions of the sets in the degrees from the class $\mathcal A$ (in Definition~\ref{D1}) and $\mathcal G$ be the class of characteristic functions of the sets in $\mathcal B$.  

We will also combine the two definitions in certain case, e.g., we will consider index sets of the form $\mathcal G [ \mathcal A]$, where $\mathcal A$ is a family of sets and $\mathcal G$ is a family of functions.

Definition~\ref{D2} also reveals the reason for the choice of the terminology ``$\mathcal G$-universal $\mathcal F$-code'': if $e \in \mathcal G [ \mathcal F]$, then $\{ e \}^g$ codes a function in $\mathcal F$, regardless of the choice of oracle $g \in \mathcal G$.  The set $\mathcal G [\mathcal F]$ is therefore the set of codes, $e$, such that the behavior of the oracle machine coded by $e$ is invariant (in the sense that it is always in $\mathcal F$) under the choice of oracle $g \in \mathcal G$.  

In what follows, we make use of both Definitions~\ref{D1} and \ref{D2}.  It will always be clear from context which one applies.

There are obvious codes in $\mathcal G [ \mathcal F]$, for many choices of $\mathcal G$ and $\mathcal F$.  For instance, if $\indset{tot}$ denotes the family of total functions, a code for a total  function which makes no oracle queries will be in the class $\mathcal G [ \indset{tot}]$, for any choice of $\mathcal G$.  Naturally, it is of interest when there are nontrivial elements of $\mathcal G [\mathcal F]$.  In this case, we give ``nontriviality'' a precise meaning with the following definition.

\begin{Definition}
We say that a $\mathcal G$-universal $\mathcal F$-code $e$ is \emph{non-trivial} if there are $g_0,g_1 \in \mathcal G$ such that $\{e\}^{g_0} \neq \{e\}^{g_1}$ and \emph{strongly non-trivial} if for all $g_0,g_1 \in \mathcal G$, $\{e\}^{g_0} \neq \{e\}^{g_1}$.
\end{Definition}

The nontrivial elements of $\mathcal G [ \mathcal F]$ are, therefore, the codes which do not simply ignore the oracle $g \in \mathcal G$.

In many case, the structure of the class $\mathcal G [ \mathcal F]$ can be quite complex.  In the present work, we quantify this by establishing the complexity of the $\mathcal G [\mathcal F]$ for various classes $\mathcal G$ and $\mathcal F$.  In several cases these index sets are arithmetic, but, in more than one instance, they are beyond hyperarithmetic.  

Letting $\indset{inf}$ denote the family of infinite c.e.~sets and $\indset{tot}$ the family of total computable functions, we have the following arithmetic complexity result.

\begin{Proposition}
The index set $\indset{inf}[\indset{tot}]$ is $\Pi^0_3$-complete.
\end{Proposition}

Similarly, letting $\indset{fin}$ denote the set of partial computable functions with finite domain, we have 

\begin{Proposition}
If $\mathcal F$ is any uniformly computable family of sets, $\mathcal F[\indset{fin}]$ has a strongly non-trivial element and is $\Pi_3^0$-complete.
\end{Proposition}

Before stating our main result, we recall a couple of standard definitions.  A total function $f: \bbn \rightarrow \bbn$ is {\em diagonally non-computable} (abbreviated, {\em dnc}) if, and only if, for each $e \in \bbn$,
\[
\{ e \} (e) \downarrow = y \implies f(e) \neq y.
\]
We let $\indset{dnc}$ denote the family of dnc functions.  A set, $A \subseteq \bbn$, is {\em immune} if it contains no infinite c.e.~set, and {\em bi-immune} if both $A$ and $\overline A$ are immune.  Let $\indset{bi}$ denote the family of bi-immune sets.  Finally, recall that a set $P \subseteq \bbn$ is $\Pi^1_1$ if it is many-one reducible to the set of Turing codes for characteristic functions of recursive trees in $\bbn^{< \omega}$, which have no infinite branches.

It is known that there is a Medvedev reduction of $\indset{dnc}$ to $\indset{bi}$ \cite{jo-le}.  In fact, the corresponding set of codes is as complicated as possible:

\begin{Theorem}
$\indset{dnc}[\indset{bi}]$ is $\Pi^1_1$-complete.
\end{Theorem}

\noindent (Note that, if $\mathcal A$ and $\mathcal B$ are hyperarithmetic, then $\mathcal B[\mathcal A]$ is at most $\Pi^1_1$.)  

This latter result is a consequence of a more general theorem which, informally, states that, if $\omega^\omega$ may be ``effectively'' embedded into a hyperarithmetic class, $\mathcal A$, and $\mathcal B$ is a hyperarithmetic tail set, Medvedev reducible to $\mathcal A$, then the index set $\mathcal A [ \mathcal B ]$ is always $\Pi^1_1$-complete.

\section{Basic facts and notation}

The next Proposition makes explicit the connection between universal codes and Medvedev reductions.

\begin{Proposition}\label{medvedev-reduct}
Suppose that $\mathcal A \subseteq \baire$ and $\mathcal B$ is a family of (possibly partial) functions on $\omega$.  $\mathcal A \leq_s \mathcal B$ if and only if $\mathcal B[\mathcal A] \neq \emptyset$.  Furthermore, $\mathcal B[\mathcal A]$ is the set of codes for reductions witnessing $\mathcal A \leq_s \mathcal B$.
\end{Proposition}

\begin{proof}
Suppose $\mathcal A$ and $\mathcal B$ are as in the statement and $\mathcal B[\mathcal A] \neq \emptyset$.  If $e \in \mathcal B[\mathcal A]$, then for each $B \in \mathcal B$, not only is $A = \{e\}^B \in \mathcal A$ c.e.~in $B$, but it is computable from $B$.  Thus, $e$ codes a turing functional that uniformly computes elements of $\mathcal A$ from elements of $\mathcal B$ -- in other words, it is a Medvedev reduction.  Conversely, if $\Psi(\sigma,x) = \{e\}^{\sigma}(x)$ witnesses $\mathcal A \leq_s \mathcal B$, then for every $B \in \mathcal B$, $\Psi(B)$ is total and hence computable in $B$.

\end{proof}

In light of this theorem, we will begin our investigation of universal codes by examining the complexity of $\mathcal B[\mathcal A]$ for some classes where it is known or trivial to see that $\mathcal A \leq_s \mathcal B$.

\begin{Proposition}
The index set $\indset{inf}[\indset{tot}]$ is $\Pi^0_3$-complete.
\end{Proposition}

\begin{proof}
First of all, to see that $\indset{inf}[\indset{tot}]$ is $\Pi^0_3$, observe that
\[
e \in \indset{inf}[\indset{tot}] \iff (\forall a,n) (a \in \indset{fin} \vee (\exists s)(\forall t \geq s) (\{e\}^{W_{a,t}}_t (n) \downarrow)).
\]

Next, to see that $\indset{inf}[\indset{tot}]$ is $\Pi^0_3$-hard, let $B$ be a fixed $\Pi^0_3$ set, with $g : \mathbb N \times \mathbb N \rightarrow \mathbb N$ a total computable function such that 
\[
x \in B \iff (\forall y) (g(x,y) \in \indset{fin}),
\]
for each $x \in \mathbb N$.  

For an infinite set $A$, let $C^*_n(A)$ denote the set of $i$ such that $i$ is the $\langle n , k\rangle$-th element of $A$, for some $k$.  If $A$ is infinite, note that each $C^*_n (A)$ is also infinite.  Let $f : \mathbb N \rightarrow \mathbb N$ be a total computable function such that, for each $e \in \mathbb N$,
\[
\{f(e)\}^A (n) \downarrow \iff (\exists m \in C^*_n (A)) ( W_{g(e,n),m+1} = W_{g(e,n),m}).
\]
In the first place, if $e \in B$, then $g(e,n) \in \indset{fin}$, for each $n \in \mathbb N$.  It follows that $\{ f(e) \}^A (n) \downarrow$, for each $n$, since $A$ is infinite and $W_{g(e,n) , m+1} = W_{g(e,n) , m}$, for all but finitely many $m$.

On the other hand, if $e \notin B$, then there exists $n$ such that $W_{g(e,n)}$ is infinite.  Let $A \subseteq \mathbb N$ be an infinite c.e.~set such that
\[
C^*_n (A) \subseteq \{ m : W_{e,m+1} \neq W_{e,m}\}.
\]
It follows that $\{ f(e) \}^A (n)$ never converges, since there is no $m = C_n^* (A)$ such that  $W_{e,m+1} = W_{e,m}$.  In other words, $f(e) \notin \indset{inf}[\indset{tot}]$.
\end{proof}

\begin{Proposition}\label{finComp}
If $\mathcal F$ is any uniformly computable family, $\mathcal F[\indset{fin}]$ has a strongly non-trivial element and is $\Pi_3^0$-complete.
\end{Proposition}

\begin{proof}
Let $\mathcal F = \{ F_0, F_1, \ldots \}$ be a uniformly computable family.  Let $p$ be a computable function such that $p(\sigma) = i$ if $i < |\sigma|$ is least such that $\sigma = F_i \upto |\sigma|$ and $p(\sigma) = -1$ if there is no such $i$.  Define a code $e$ such that
\[
\{ e \}^A(x) = \begin{cases}
p(A\upto x) &\mbox{if $p(A\upto x) \neq p(A\upto(x-1)$}\\

\uparrow & \mbox{otherwise}
\end{cases}.
\]
If $A \in \mathcal F$, then $\{e\}^A$ has finite domain and if $A,B \in \mathcal F$ with $A \neq B$, then $\{e\}^A \neq  \{e\}^B$.

We now prove that $\mathcal F[\indset{fin}]$ is $\Pi_3^0$-complete.  To see that $\mathcal F[\indset{fin}]$ is $\Pi_3^0$, observe that
\[
\mathcal F[\indset{fin}] = \Big\{ e : (\forall i)(\exists b)(\forall x > b, s)\big( \{e\}_s^{F_i}(x)\uparrow \big) \Big\}.
\]

We use a movable markers argument to prove hardness.  Let $f: \bbn^2 \rightarrow \{0,1\}$ be a computable function such that $f_i(x) = f(i,x)$ is the characteristic function of $F_i$.  Define a computable function $g$ such that $g(e,n,s) = y$ if and only if $y$ is the $n^{th}$ element of $\overline{W}_{e,s}$.  Define a computable function $h$ such that $\{h(e)\}^{\sigma}(x) \downarrow = 1$ if
\begin{enumerate}
\item $|\sigma| \geq x$ and
\item if $i$ is least such that $\sigma \upto x = f_i \upto x$, then there is an $s > x$ such that $g(e,i,s) \neq g(e,i,x)$.
\end{enumerate}

If $e \in \indset{cof}$, then for all $i > \max(\overline{W}_e)$ and $x \in \bbn$, there is a $y > x$ such that $g(e,i,y) \neq g(e,i,x)$.  Thus, for all $i > \max(\overline{W}_e)$, $\{h(e)\}^{F_i}$ is total and $h(e) \not\in \mathcal F[\indset{fin}]$.  On the other hand, if $e \in \indset{coinf}$, then for all $i$ $\lim_{x \rightarrow \infty} g(e,i,x)$ exists.  In other words, for every $i \in \bbn$ there is an $x$ such that for all $y > x$, $g(e,i,y) = g(e,i,x)$.  We conclude that $\{h(e)\}^{F_i}$ is finite for all $i \in \bbn$ and $h(e) \in \mathcal F[\indset{fin}]$.  Thus, $h$ witnesses the desired result: $\mathcal F[\indset{fin}]$ is $\Pi_3^0$-hard.

\end{proof}

Note that while an oracle drawn from a uniformly computable family does not confer additional computational power it does affect the output of an oracle program.

\begin{Theorem}
$\Delta_2^0[\indset{fin}]$ is $\Pi_4^0$-complete and contains a strongly non-trivial element.
\end{Theorem}

\begin{proof}
First, we prove that $\Delta_2^0[\indset{fin}]$ is $\Pi_4^0$.  To this end, we say that $d \in \omega$ is a {\em $\Delta^0_2$ code} if $\{ d \}$ is a total function and, for each $x \in \omega$
\[
\lim_{s \rightarrow \infty} \{ d \} (\langle s,x \rangle)
\]
exists, where $\langle \cdot , \cdot \rangle$ is a fixed computable pairing function.  Note that the predicate ``$d$ is a $\Delta^0_2$ code'' is $\Pi^0_3$.  Supposing that $d$ is a $\Delta^0_2$ code, define the functions $d_t$, for $t \in \omega$, by 
\[
d_t (x) = \{ d \} (\langle t , x \rangle)
\]
and $d_*$ by
\[
d_* ( x) = \lim_{t \rightarrow \infty} d_t (x).
\]

Define predicates $\Gamma$ and $\Sigma$ as follows:
\begin{align*}
\Gamma (e,d,x) &\iff (\forall k) (\exists t) (\mathrm{use} (\{ e\}^{d_t}_t (x)) \geq k)\\
&\iff \{ \mathrm{use} (\{ e \}^{d_t}_t (x)) : t \in \omega\} \mbox{ is unbounded}.
\end{align*}
and
\[
\Sigma (e,d,x) \iff (\forall s ) (\exists t \geq s) ( \{ e \}^{d_t}_t (x) \uparrow).
\]
It will follow that $e \in \Delta^0_2$ if, and only if,
\begin{equation}\label{E1}
(\forall d) \left( d \mbox{ is a $\Delta^0_2$ code }   \implies (\exists m) (\forall x \geq m)(\Gamma(e,d,x) \vee \Sigma(e,d,x))\right).
\end{equation}
In particular, this will show that $\Delta^0_2 [\indset{fin}]$ is $\Pi^0_4$.

Indeed, suppose that $d$ is a $\Delta^0_2$ code and $x \in \omega$.  Suppose first that $\{ e \}^{d_*} (x)$ converges.  Let $k_0 = \mathrm{use} (\{ e \}^{d_*} (x))$ and let $t_0$ be large enough that, for each $t \geq t_0$, 
\[
d_t \upto k_0 = d_* \upto k_0.
\]
It follows that, for each $t \geq t_0$, 
\[
\mathrm{use}(\{ e \}^{d_t}_t (x)) \leq k_0
\]
and, hence, 
\[
\{ \mathrm{use} (\{ e \}^{d_t}_t (x)) : t \in \omega\}
\]
is bounded, i.e., $\neg \Gamma (e,d,x)$.  Also, if $s_0$ is large enough that $\{ e \}^{d_*}_{s_0} (x)$ has converged and $t\geq \max\{ s_0 , t_0 \}$, then $\{ e \}^{d_t}_t (x)$ converges.  Thus, $\neg \Sigma(e,d,x)$.  As $d$ and $x$ were arbitrary, this establishes the ``$\implies$'' part of (\ref{E1}).

On the other hand, suppose that $x \in \omega$ and $d$ is a $\Delta^0_2$ code such that the computation $\{ e \}^{d_*}(x)$ diverges.  There are two cases.  In the first place, suppose that $\mathrm{use} (\{ e \}^{d_*}_t)$ is unbounded, as $t \rightarrow \infty$.  Fix $k \in \omega$ and let $s$ be such that
\[
u = \mathrm{use} (\{ e \}^{d_*}_s (x) ) \geq k.
\]
Let $t_0$ be such that, for each $t \geq t_0$,
\[
d_t \upto u = d_* \upto u.
\]
If $t = \max\{ s,t_0 \}$, then 
\[
\mathrm{use} (\{ e \}^{d_t}_t (x)) \geq \mathrm{use} (\{ e \}^{d_t}_s (x))  = \mathrm{use} (\{ e \}^{d_*}_t (x)) \geq k.
\]
As $k$ was arbitrary, it follows that $\{ \mathrm{use} (\{ e \}^{d_t}_t : t \in \omega\}$ is unbounded, i.e., $\Gamma (e,d,x)$.  

Secondly, suppose that $\{ e \}^{d_*} (x)$ diverges, but the use of the computation is bounded, say
\[
\mathrm{use} (\{ e \}^{d_*}_t (x)) \leq u, 
\]
for all $t \in \omega$.  Let $t_0$ be such that
\[
d_t \upto u = d_* \upto u,
\]
for all $t \geq t_0$.  In particular, for each $t \geq t_0$, the computation $\{ e \}^{d_t}_t (x)$ is equivalent to the computation $\{ e \}^{d_*}_t (x)$.  Hence, $\{ e \}^{d_t}_t (x)$ diverges for all $t \geq t_0$.  Hence, $\Sigma (e,d,x)$ holds, since $d$ and $x$ were arbitrary.

Combining the two cases above establishes the ``$\Longleftarrow$'' part of (\ref{E1}).

Fix a $\Pi_4^0$ predicate, $Q$, and a computable function, $g$, such that
\[
Q(e) \leftrightarrow (\forall x)(\exists y)\big[ g(e,x,y) \in \indset{inf} \big]
\]
Define a computable function, $h$, such that 
\[
\{h(e)\}_s^A(n) = \begin{cases}
\langle A \upto n \rangle &\mbox{if }(\forall j \leq n)\big[ \min(C_{j+1}(A)) < s \\
	&\wedge\min(C_{j+1}(A)) > \max(W_{g(e,i,j), s}) \big]\\
\uparrow &\mbox{otherwise}
\end{cases},
\]
where $i = \min(C_0(A))$ and $\{h(e)\}^A(n) \uparrow$ for all $n\in\mathbb N$ if $C_0(A) = \emptyset$.

If $C_0(A) = \emptyset$, then $W_e = \emptyset$.  If $Q(e)$ and $C_0(A) \neq \emptyset$, let $i = \min(C_0(A))$.  Because $Q(e)$, there is a $j$ such that $g(e,i,j) \in \indset{inf}$.  Hence, for all but finitely many $s$, either $\min(C_{j+1}(A)) \geq s$ or $\min(C_{j+1}(A)) < \max(W_{g(e,i,j), s})$ and $W_e$ is finite.  We conclude that $h(e) \in \Delta^2_0[\indset{fin}]$.

Now suppose that ${}^{\neg}Q(e)$ and let $x$ be such that $(\forall y)[g(e,x,y)\in \indset{fin}]$.  Since the maximum of each $W_{g(e,x,y)}$ for $y\in\mathbb N$ can be found in the limit, there is a $\Delta^2_0$ oracle, $A$, such that $C_{j}(A) \neq \emptyset$ for all $j \in\mathbb N$, $\min(C_0(A)) = x$ and $\min(C_{j+1}(A)) > \max(W_{g(e,x,j)})$ for all $j \in \mathbb N$.  Since $W^A_{h(e)}$ is infinite and $A \in \Delta^2_0$, $h(e) \not\in \Delta^2_0[\indset{fin}]$.

Thus, $h$ reduces $Q$ to $\Delta^2_0[\indset{fin}]$.

\end{proof}

Before stating the next proposition, we recall the following standard definitions.  

\begin{Definition}\label{dnc def}
A function, $f$, is said to be \emph{diagonally non-computable} if $f(e) \neq \{e\}(e)$ whenever $\{e\}(e)\downarrow$.  Let $\indset{dnc} = \{ f \in \omega^\omega : f \mbox{ is DNC} \}$.  

A function, $f$, is {\em $n$ diagonally non-computable ($n$-DNC)} if it is diagonally non-computable and, additionally, $f(e) \leq n$, for each $e$.  Let $\indset{dnc}_n = \{ f \in \omega^\omega : f \mbox{ is $n$-DNC} \}$.
\end{Definition}

\begin{Theorem}
$\indset{dnc}_{n}[\indset{dnc}_{n+1}]$ is $\Pi_2^0$-complete for all $n \geq 2$.
\end{Theorem}

\begin{proof}
$e \not\in \indset{dnc}_{n}[\indset{dnc}_{n+1}]$ if and only if 
\[
(\exists \sigma \in \bairestr)\Big( \mbox{$\sigma$ is a $\indset{dnc}_{n}$ string} \wedge \mbox{$\{e\}^{\sigma}\upto |\sigma|$ is not a $\indset{dnc}_{n+1}$ string} \Big).
\]
The statement ``$\sigma$ is a $\indset{dnc}_{n}$ string'' is equivalent to
\[
(\forall x < |\sigma|)\Big( (\forall s)\big(\{x\}_s(x)\uparrow\big) \vee (\exists s)\big(\{x\}_s(x)\downarrow \neq \sigma(x) < n \big) \Big),
\]
and the statement ``$\{e\}^{\sigma}\upto |\sigma|$ is not a $\indset{dnc}_{n+1}$ string'' is equivalent to
\begin{align*}
(\exists x < |\sigma|)\Big( (\forall s)\big( \{e\}^{\sigma}_s(x) \uparrow \big) \vee (\exists s,y) \big( &\{e\}^{\sigma}_s(x) \downarrow = y \\
&\wedge (y = \{x\}_s(x) \downarrow \vee y \geq n + 1) \big) \Big).
\end{align*}
Thus, ``$e \not\in \indset{dnc}_{n}[\indset{dnc}_{n+1}]$'' is $\Sigma_2^0$, i.e.,  $\indset{dnc}_{n}[\indset{dnc}_{n+1}]$ is $\Pi_2^0$.

Define a computable function $f$ such that $\{f(e)\}^{\sigma}\upto |\sigma| = \sigma \upto |W_{e,|\sigma|}|$.  If $e \in \indset{inf}$, then $\{f(e)\}^g = g$ for all $g \in \baire$.  If $e \in \indset{fin}$, then the domain of $\{f(e)\}^g$ is finite for all $g \in \baire$.  Since every $\indset{dnc}_n$ is also $\indset{dnc}_{n+1}$, $f(e) \in \indset{dnc}_{n}[\indset{dnc}_{n+1}]$ if $e \in \indset{inf}$ and not otherwise, showing that $\indset{dnc}_{n}[\indset{dnc}_{n+1}]$ is $\Pi_2^0$-hard.

\end{proof}

\section{$\Pi^1_1$-completeness}

In what follows we give a general, but somewhat technical, theorem which implies that a number of natural index sets of the form $\mathcal A [ \mathcal B]$ are $\Pi^1_1$-complete.

\begin{Definition}
Let $c$ be a finite ordinal or $\omega$.  We say that a function $f: \omega^\omega \rightarrow c^\omega$ is a {\em $\Delta^0_2$ embedding} if, and only if, there is a uniformly computable sequence, $f_s : \omega^{<\omega} \rightarrow c^{<\omega}$, of total recursive functions such that the following conditions hold:
\begin{enumerate}
	\item $f_s (\alpha) \prec f_s(\beta)$ if, and only if, $\alpha \prec \beta$, i.e., each $f_s$ is a {\em $\prec$-isomorphism}
	\item $\lim_s f_s (\alpha)$ exists, for each $\alpha \in \omega^{<\omega}$
	\item $f(x) = \bigcup_n \lim_s f_s (x \upto n)$
\end{enumerate}
Note that the union $\bigcup_n \lim_s f_s (x \upto n)$ is a well-defined function on $\omega$ since the $f_s$ all preserve proper extension
\end{Definition}

\begin{Theorem}\label{pi one one}
Suppose that $\mathcal A  \subset c^\omega$ (where $c$ is a finite ordinal or $\omega$) and $\mathcal B \subseteq 2^\omega$ are hyperarithmetic, with $\mathcal B$ a tail set containing no finite sets and no cofinite sets, and there is a $\Delta^0_2$ embedding $f: \omega^\omega \rightarrow c^\omega$ such that the range of $f$ is relatively closed in $\mathcal A$.  If there is a Medvedev reduction, $\Phi$, of $\mathcal B$ to $\mathcal A$, then the set of $\mathcal A$-universal $\mathcal B$-codes is $\Pi^1_1$-complete.
\end{Theorem}

\begin{proof}
Let $\mathcal A$, $\mathcal B$, $\Phi$ and $f$ be as above, with $f_s : \omega^{<\omega} \rightarrow c^{<\omega}$ witnessing that $f$ is a $\Delta^0_2$ embedding.  For convenience, define $f_* : \omega^{<\omega} \rightarrow c^{<\omega}$ by $f_* (\alpha) = \lim_s f_s (\alpha)$ and observe that, by the properties of the $f_s$, the map $f_*$ is also a $\prec$-isomorphism.  

In the first place, it follows from the definition of $\mathcal A [\mathcal B]$ and the fact that $\mathcal A$ and $\mathcal B$ are both hyperarithmetic that $\mathcal A [ \mathcal B ]$ is $\Pi^1_1$.  To show that $\mathcal A [ \mathcal B ]$ is $\Pi^1_1$-hard, it will suffice to reduce the $\Pi^1_1$-complete set $\indset{NoPath}$ (see Corollary IV.2.16 from \cite{odifreddi}) to $\mathcal A [ \mathcal B ]$, where $\indset{NoPath}$ is the set of codes for recursive trees with no infinite branches.  More precisely, if $\rho_0 , \rho_1 , \ldots$ is a recursive enumeration of $\omega^{<\omega}$ and $e \in \omega$, with $\{ e \}$ total, define $T_e$ to be the tree generated by 
\[
\{ \rho_i : \{ e \} (i) \downarrow = 1\}.
\]
Let
\[
\indset{NoPath} = \{ e : \{ e \} \mbox{ is total and } [T_e] = \emptyset\}.
\]

With this in mind, we will define Turing functionals $\Phi_e$ such that $\Phi_e$ is a Medvedev reduction of $\mathcal B$ to $\mathcal A$, if, and only if, $e \in \indset{NoPath}$.  For each $e \in \omega$ and $\sigma \in \omega^{<\omega}$, let
\[
n_{e , \sigma} = \max \{ \min \{ s , |\beta| \} : (\exists \alpha \in T_{e , |\sigma|}) (\beta \preceq f_s (\alpha) \wedge\sigma \succeq \beta\}
\]
and 
\[
m_{e , \sigma} = \max \{ m : (\forall i \leq m) ( \{ e \}_{|\sigma|} (\tau_i) \downarrow)\},
\]
where $\tau_0 , \tau_1 , \ldots$ is a fixed recursive enumeration of $c^{<\omega}$.  Finally, define 
\[
\Phi_e (\sigma) = (\Phi(\sigma) \cup [0,n_{e , \sigma}]) \cap [0 , m_{e , \sigma}].
\]

In the first place, if $\{ e\}$ is not total, then $m_{e,\sigma}$ is bounded as $\sigma$ varies over $\omega^{<\omega}$ and, hence, $\lim_{n \rightarrow \infty} \Phi_e (X \upto n)$ is a finite set (consequently, not in $\mathcal B$), for each $X \in \omega^\omega$.  It follows that $\Phi_e$ is not a Medvedev reduction of $\mathcal B$ to $\mathcal A$. 

Therefore, assume that $\{ e \}$ is total, i.e., $m_{e , \sigma} \rightarrow \infty$, as $|\sigma| \rightarrow \infty$.  

Suppose first that $[T_e] \neq \emptyset$, with $X \in [T_e]$.  The family $\mathcal B$ contains no cofinite sets (in particular, $\mathcal B$ does not contain $\omega$) and, hence, to show that $\Phi_e$ is not a Medvedev reduction of $\mathcal B$ to $\mathcal A$, it will suffice to show that $\Phi_e (X) = \bigcup_n \Phi_e (X \upto n) = \omega$.  In turn, it will be enough to show that
\[
\lim_{\substack{\sigma \prec X\\ |\sigma| \rightarrow \infty}} n_{e , \sigma} = \infty.
\]
Indeed, fix $n_0 \in \omega$ and let $\sigma_0 \prec f(X)$ be long enough that there exists an $\alpha_0 \in T_{e , |\sigma_0|}$ with $|f_* (\alpha_0)| \geq n_0$.  Next, let $s_0 \geq n_0$ be large enough that $f_s (\alpha) = f_* (\alpha)$, for each $s \geq s_0$ and $\alpha \preceq \alpha_0$.  Letting $\beta = f_* (\alpha_0) = f_{s_0} (\alpha_0)$, it follows that 
\[
n_{e , \sigma} \geq \min \{ s_0 , \beta \} \geq n_0,
\]
for every $\sigma \succeq \sigma_0$, by the definition of $n_{e , \sigma}$.

Next, assume that $[T_e] = \emptyset$.  Since $\mathcal B$ is a tail set, to show that $\Phi_e$ is Medvedev reduction of $\mathcal B$ to $\mathcal A$, it will suffice to show that, for each $Y \in \mathcal A$, the set $\{ n_{e , \sigma} : \sigma \prec Y\}$ is bounded and, hence, $\Phi_e(Y) $ differs only finitely from $\Phi(Y)$, for each $Y \in \mathcal A$.  

Indeed, fix $Y \in \mathcal A$.  First, suppose that $Y = f(X)$, for some $X \in \omega^\omega$.  Let $\alpha_0 \prec X$ be longest such that $\alpha_0 \in T_e$.  Let $s_0$ be large enough that $f_s (\alpha) = f_* (\alpha)$, for each $s \geq s_0$ and $\alpha \preceq \alpha_0$.  Fix $\sigma \prec Y = f(X)$, with $\sigma \succeq f_*(\alpha_0)$.  Suppose that $s \geq s_0$, $\alpha \in T_e$ and $\beta \in c^{<\omega}$, with $\beta \preceq f_s (\alpha)$ and $\sigma \succeq \beta$.  If $|\beta| > |f_* (\alpha_0)|$, then $f_s (\alpha) \succ f_*(\alpha_0) = f_s(\alpha_0)$.  Hence, $\alpha \succ \alpha_0$, since $f_s$ is a $\prec$-isomorphism.  This is a contradiction, since no extension of $\alpha_0$ is in $T_e$.  It follows that $|\beta| \leq |f_*(\alpha_0)|$.  Consequently, if $\alpha \in T_e$, $s \in \omega$ and $\beta , \sigma \in c^{<\omega}$ are such that $\beta \preceq f_s (\alpha)$ and $\sigma \succeq \beta$, then either $|\beta| \leq |f_* (\alpha_0)|$ or $s < s_0$.  It follows that 
\[
n_{e , \sigma} \leq \max \{ s_0 - 1 , |f_*( \alpha_0 )| \}, 
\]
for any $\sigma \prec Y$.

Finally, assume that $Y \in \mathcal A$, but $Y \neq f(X)$, for every $X \in \omega^\omega$.  Since the range of $f$ is relatively closed in $\mathcal A$, choose $\sigma_0\prec Y$ longest such that $\sigma_0 \prec f(X)$, for some $X \in \omega^\omega$.  Let $\alpha_0 \in \omega^{<\omega}$ be such that $f_*(\alpha_0) \succeq \sigma_0$.  Let $s_0$ be large enough that $f_s (\alpha) = f_* (\alpha)$, for all $s \geq s_0$ and $\alpha \preceq \alpha_0$.  Fix $\sigma \prec Y$, with $|\sigma| \geq |\sigma_0|$.  If $s$, $\beta$ and $\alpha$ are such that $\beta \preceq f_s (\alpha)$ and $\sigma \succeq \beta$, then either $\beta \preceq \sigma_0$ or $s < s_0$.  Thus, 
\[
n_{e , \sigma} \leq \max \{ s_0 - 1, |\sigma_0|\}
\]
and it follows that
\[
\{ n_{e , \sigma} : \sigma \prec Y\}
\]
is bounded.  This completes the proof.
\end{proof}

\begin{Definition} We define several forms of immunity and associated index sets.
\begin{enumerate}
\item $A$ is \emph{immune} if $A$ is infinite and contains no infinite c.e.~set.

\item $A$ is \emph{bi-immune} if $A$ and $\overline{A}$ are both immune.

\item $\indset{im} = \{ A \subset \mathbb N : A \mbox{ is immune} \}$.

\item $\indset{bi} = \{ A \subset \mathbb N : A \mbox{ is bi-immune} \}$.

\end{enumerate}
\end{Definition}

\begin{Corollary}\label{dnc-bi}
The index sets $\indset{dnc}[\indset{im}]$ and $\indset{dnc}[\indset{bi}]$ are both $\Pi^1_1$-complete. 
\end{Corollary}

\noindent See Definition~\ref{dnc def} for the definition of $\indset{dnc}$.

\begin{proof}
An examination of the relevant definitions reveals that the sets {\sc dnc}, {\sc im} and {\sc bi} satisfy the required topological and definability properties to apply Theorem~\ref{pi one one}.  Since there are known to be Medvedev reductions of $\indset{im}$ and $\indset{bi}$ to $\indset{dnc}$, it will suffice to define a function $f: \omega^\omega \rightarrow \omega^\omega$ which is a $\Delta^0_2$ embedding into $\indset{dnc}$.  

Let $\langle \cdot \rangle : \omega^{<\omega} \rightarrow \omega$ be a recursive coding of finite strings.  For each $\alpha \in \omega^{<\omega}$ and $n < |\alpha|$, define 
\[
f_s (\alpha) (n) = \begin{cases}
2 \langle \alpha \upto n \rangle &\mbox{if } \{n\}_s(n) \neq 2 \langle \alpha \upto n \rangle\\
1 + 2 \langle \alpha \upto n \rangle &\mbox{otherwise}
\end{cases}.
\]
Note that $f_s (\alpha)$ is a string of integers of length $|\alpha|$ and, moreover, that each $f_s$ is a $\prec$-isomorphism.  For $X \in \omega^\omega$, let $f(X)$ be as in the definition of a $\Delta^0_2$ embedding, i.e., 
\[
f(X) = \bigcup_n \lim_{s\rightarrow \infty} f(X \upto n).
\] 
First of all, note that the range of $f$ is the closed subset of $\omega^\omega$, consisting of those $X$ such that, for each $n \in \omega$, 
\[
X(n) = \begin{cases}
2\langle X \upto n \rangle &\mbox{if } \{ n \} (n) \downarrow \neq 2 \langle X \upto n \rangle\\
1 + 2 \langle X \upto n \rangle &\mbox{otherwise}.
\end{cases}
\]
Furthermore, note that $f(X) \in \indset{dnc}$, for each $X \in \omega^\omega$.  We may now apply Theorem~\ref{pi one one} to conclude that $\indset{dnc}[ \indset{bi} ]$ is $\Pi^1_1$-complete.
\end{proof}

With a couple more definitions, we will be able to state another corollary of Theorem~\ref{pi one one}

\begin{Definition}
Let $\mathcal P_{\mathsf{fin}} (\omega)$ denote the family of finite subsets of $\omega$.  A {\em canonical numbering} is a total computable function $H : \omega \rightarrow \mathcal P_{\mathsf{fin}} (\omega)$ such that 
\begin{enumerate}
\item each finite set is in the range of $H$,
\item the predicate ``$x \in H(e)$'' is computable, and
\item the function $e \mapsto \max H(e)$ is computable.
\end{enumerate}
\end{Definition}

Identifying $\mathcal P_{\mathsf{fin}} (\omega)$ with $2^{<\omega}$, we could alternatively characterize a canonical numbering as a total computable function $H : \omega \rightarrow 2^{< \omega}$ such that each finite set is in the range of $H$.

\begin{Definition}\cite{bkk-rodfest}
A infinite set $R \subseteq \omega$ is {\em canonically immune} if, and only if, there is a total computable function $h$ such that, for each canonical numbering $H$, and all but finitely many $e \in \omega$, 
\[
H(e) \subseteq R \implies |H(e)| \leq h(e).
\]
\end{Definition}

\begin{Definition}\cite{snr-def}
A function $f: \omega \rightarrow \omega$ is {\em strongly non-recursive} if, and only if, for each total computable function $h:\omega \rightarrow \omega$, one has $f(n) \neq h(n)$, for all but finitely many $n$.
\end{Definition}

\begin{Corollary}
The index set {\sc ci[\sc snr] }is $\Pi^1_1$-complete.
\end{Corollary}

\begin{proof}
In what follows, we freely identify an element of $2^\omega$ or $2^{<\omega}$ with the subset of $\omega$ of which it is the characteristic function.  

In the first place, an inspection of the proof Theorem 5.5 in \cite{bkk-rodfest} reveals that there is a Medvedev reduction of {\sc snr} to {\sc ci}.  It follows from the definitions of {\sc ci} and {\sc snr} that the requisite topological and definability properties are satisfied in order to apply Theorem~\ref{pi one one}.  All that remains is to define a $\Delta^0_2$ embedding into {\sc ci} with a relatively closed range.  

To this end, we begin by defining a universal function for canonical numberings.  Let $\varphi : \omega^2 \rightarrow 2^{<\omega}$ be a universal partial recursive function.  Define 
\[
D_{r,s} (e) = \begin{cases}\varphi (r,e) &\mbox{if $\varphi (r,e)$ converges within $s$ stages},\\ \langle \emptyset \rangle &\mbox{otherwise}. \end{cases}
\]
Let $D_r (e) = \lim_{s\rightarrow \infty} D_{r,s} (e)$.  It follows that each canonical numbering appears as $D_r$, for some $r$, though not every $D_r$ is a canonical numbering.  The function $D$ is itself limit computable.

Define 
\[
F_n = \bigcup_{r,e \leq n} D_r (e)
\]
and note the sequence, $F_n$, of finite sets is uniformly limit computable.  Since every finite set is contained in some $F_n$, it follows that there exist $n_0 < n_1 < \ldots$ and $x_0 < x_1 < \ldots$ such that, for each $i$, 
\[
F_{n_i} \setminus F_{n_i - 1} \neq \emptyset
\]
and 
\[
x_i \in F_{n_i} \setminus F_{n_i - 1}.
\]
The sequence, $(x_i)_{i \in \omega}$, may be chosen to be strictly increasing and limit computable.  We may, therefore, take a computable sequence, $(x_{i,s})_{i,s \in \omega}$ such that, for each $i$,
\[
x_i = \lim_{s \rightarrow \infty} x_{i,s}
\]
and, for fixed $s$, the $x_{i,s}$ are all distinct.

For each $s \in \omega$ and $\alpha \in \omega^{<\omega}$, define $f_s (\alpha) \in 2^{<\omega}$ to have length 
\[
x_{(\alpha(0) + \ldots + \alpha (|\alpha| - 1) + |\alpha|-1) , s} + 1
\]
and be such that
\[
f_s (\alpha) (j) = \begin{cases}1 &\mbox{if } (\exists i,p \in \omega) ( p < |\alpha| \wedge j = x_{i,s}  \wedge i = \alpha(0) + \ldots + \alpha(p) + p),\\ 0 &\mbox{otherwise}.\end{cases}
\]
It follows that each $f_s$ is a $\prec$-isomorphism.  Let $f : \omega^\omega \rightarrow \textsc{ci}$ be as in the definition of a $\Delta^0_2$ embedding.

It remains to verify that each $f(X)$ is canonically immune.  Indeed, suppose that $H : \omega \rightarrow 2^{<\omega}$ is a canonical numbering, with $H = D_r$.  For each $e \geq r$ observe that $D_r(e) \subseteq F_e$ and, hence, for any $X \in \omega^\omega$,
\[
f(X) \cap D_r (e) \subseteq \{ x_0 , \ldots , x_e\}.
\]
In particular, $|f(X) \cap D_r (e)| \leq e+1$.  As $H$ was arbitrary, it follows that each $f(X)$ is canonically immune, witnessed by the computable function $h(e) = e+1$.

Finally, to see that the range of $f$ is relatively closed, simply observe that the range of $f$ is the intersection of {\sc ci} with the closed set
\[
\{ Y \in 2^\omega : (\forall j) (Y(j) \neq 0 \implies (\exists i) ( j = x_i )) \}.
\]  
This completes the proof.
\end{proof}

In Theorem~\ref{pi one one}, we required that the $\Delta^0_2$ embedding have relatively closed range in the class $\mathcal A$.  In fact, we can achieve the same result if we require that the map, $f$, have relatively $\Pi^0_2$ range in $\mathcal A$. 

\begin{Theorem}\label{pi 2 version}
Suppose that $\mathcal A \subseteq c^\omega$ (where $c$ is a finite ordinal or $\omega$) and $\mathcal B \subseteq 2^\omega$ are hyperarithmetic, with $\mathcal B$ a tail set containing no finite sets and no cofinite sets, and there is a $\Delta^0_2$ embedding $f: \omega^\omega \rightarrow c^\omega$ such that the range of $f$ is the intersection of a $\Pi^0_2$ class, $P$, with $\mathcal A$.  If there is a Medvedev reduction, $\Phi$, of $\mathcal B$ to $\mathcal A$, then the set of $\mathcal A$-universal $\mathcal B$-codes is $\Pi^1_1$-complete.
\end{Theorem}

Note that this result is neither a generalization of Theorem~\ref{pi one one} nor {\em vice versa}, since not every closed set is a $\Pi^0_2$ class and not every $\Pi^0_2$ class is closed.  Before proceeding with the proof of Theorem~\ref{pi 2 version}, we recall the definition of a $\Pi^0_2$ class.

\begin{Definition}\label{def of pi 2}
Let $c$ be a finite ordinal or $\omega$.  A class, $P \subseteq c^\omega$, is a {\em $\Pi^0_2$ class} if, and only if, there is a total computable function $h : \omega^2 \rightarrow c^{<\omega}$ such that, for each $X \in c^\omega$, 
\[
X \in P \iff (\forall n) (\exists s) (h(n,s) \prec X).
\]
\end{Definition}

\begin{proof}[Proof of Theorem~\ref{pi 2 version}]
Let $\mathcal A$, $\mathcal B$, $f$, $P$ and $\Phi$ be as in the statement of Theorem~\ref{pi 2 version}.  Let $f_s$ be as in the definition of a $\Delta^0_2$ embedding, witnessing that $f$ is such an embedding.  Again, let $f_*$ be the pointwise limit of the sequence $(f_s)_{s \in \omega}$.  Finally, let $h : \omega^2 \rightarrow c^{<\omega}$ be a total recursive function, as in Definition~\ref{def of pi 2}, witnessing that $P$ is a $\Pi^0_2$ class.  For convenience, we write
\[
U_{n,s} = \{ \sigma \in c^{<\omega} : (\exists t \leq s)(h(n,t) \preceq \sigma)\}.
\]
Each $U_{n,s}$ is, by definition, closed under extension.  With this notation, for each $X \in c^\omega$,
\[
X \in P \iff (\forall n) (\exists s) (\mbox{an initial segment of $X$ is in $U_{n,s}$}).
\]

As in the proof of Theorem~\ref{pi one one}, we reduce the $\Pi^1_1$-complete set $\indset{NoPath}$ to $\mathcal A [ \mathcal B]$.  With respect to the coding of recursive trees, we adopt notation from the proof of Theorem~\ref{pi one one}.

What follows is similar in character to the proof of Theorem~\ref{pi one one}, with the addition of some refinements to accomodate the fact that the range of $f$ may not be closed.  Given $e \in \omega$ and $\sigma \in c^{<\omega}$, let 
\begin{align*}
n_{e ,\sigma} = \max\{ \min \{ s , |\beta|, n\} : (\exists \alpha &\in T_{e , |\sigma|})( \beta \preceq f_s (\alpha) \wedge \sigma \succeq \beta\\
&  \wedge (\forall k \leq n) (\beta \in U_{k,|\sigma|}))\}.
\end{align*}
As before, let
\[
m_{e , \sigma} = \max \{ m : (\forall i \leq m) ( \{ e \}_{|\sigma|} (\tau_i) \downarrow\},
\]
where $\tau_0 , \tau_1 , \ldots$ is a fixed recursive enumeration of $c^{<\omega}$.  Define 
\[
\Phi_e (\sigma) = (\Phi(\sigma) \cup [0,n_{e , \sigma}]) \cap [0 , m_{e , \sigma}].
\]
As in the proof of Theorem~\ref{pi one one}, if $\{ e \}$ is not total, then $m_{e , \sigma}$ is bounded as $\sigma$ varies over $c^{<\omega}$ and, consequently, $\Phi_e (Y)$ is the characteristic function of a finite set, for every $Y \in c^\omega$.  Thus, $\Phi_e$ is not a Medvedev reduction of $\mathcal B$ to $\mathcal A$, since $\mathcal B$ contains no finite sets.

Therefore, suppose that $\{ e \}$ is total and $T_e$, the recursive tree coded by $e$, has an infinite branch, $X \in [T_e]$.  To show that $\Phi_e$ is not a Medvedev reduction of $\mathcal B$ to $\mathcal A$, it will suffice to show that
\[
\lim_{k \rightarrow \infty} n_{e , X \upto k} = \infty
\]
as $\Phi_e$ will then be the characteristic function of $\omega$.  Indeed, fix $n_0 \in \omega$.  Let $\sigma_0 \prec f(X)$, with $|\sigma_0|$ such that
\begin{itemize}
	\item $|\sigma_0| \geq n_0$ and
	\item $\sigma_0 \in U_{n_0 , |\sigma_0|}$.
\end{itemize}
Now choose $\sigma_1$ such that
\begin{itemize}
	\item $\sigma_0 \preceq \sigma_1 \prec f(X)$ and 
	\item there exists $\alpha_0 \in T_{e , |\sigma_1|}$ such that $\sigma_0 \prec f_* (\alpha_0) \prec f(X)$.
\end{itemize}
Finally, let $s_0$ be such that
\begin{itemize}
	\item $s_0 \geq n_0$ and
	\item $f_{s_0} (\alpha_0) = f_* (\alpha_0)$.
\end{itemize}
It follows from the definition of $n_{e , \sigma}$ that, for each $\sigma \succeq \sigma_1$
\[
n_{e , \sigma} \geq \min\{ s_0 , f_* (\alpha_0) , n_0\} \geq n_0.
\]

Suppose now that $[T_e] = \emptyset$.  We must show that $\Phi (Y) = \lim_{j \rightarrow \infty} \Phi(Y \upto j) \in \mathcal B$, for each $Y \in \mathcal A$.  In the first place, suppose that $Y \in \mathrm{range} (f)$, say with $Y = f(X)$.  To show that $\Phi(Y) \in \mathcal B$, it will suffice to show that 
\[
\{ n_{e , \sigma} : \sigma \prec f(X) \}
\]
is bounded.  Let $\alpha_0$ be longest with $\alpha_0 \in T_e$ and let $s_0$ be such that, for each $s \geq s_0$ and $\alpha \preceq \alpha_0$, we have $f_s (\alpha) = f_* (\alpha)$.  Suppose now that $\sigma \prec f(X)$, with $f_* (\alpha_0) \prec \sigma$, and $s , \beta$ are such that there exists $\alpha \in T_e$, with $f_* (\alpha_0) \prec \beta \preceq f_s (\alpha)$ and $\beta \preceq \sigma$.  If $s \geq s_0$, we have $\alpha_0 \prec \alpha$, since $f_s (\alpha_0) = f_* (\alpha_0)$ and $f_s$ is a $\prec$-isomorphism.  Hence, if $\sigma \prec f(X)$ and $s,\beta$ are such that there exists $\alpha \in T_e$ with $\beta \prec f_s (\alpha)$ and $\sigma \succeq \beta$, then either $s < s_0$ or $|\beta| \leq f_* (\alpha_0)$.  It follows that
\[
n_{e , \sigma} \leq \min \{ s_0 - 1 , |f_* (\alpha_0)|\}
\]
for $\sigma \prec f(X)$.

Finally, suppose that $Y \in \mathcal A$, but $Y \notin \mathrm{range} (f)$.  Again, we will see that
\[
\{ n_{e , \sigma} : \sigma \prec f(X) \}
\]
is bounded.  Since $Y \notin \mathrm{range} (f)$, let $n_0$ be such that, for every $n > n_0$, no initial segment of $Y$ lies in $\bigcup_t U_{n,t}$.  It follows from the definition of $n_{e , \sigma}$ that, for any $\sigma \prec Y$, we have $n_{e , \sigma} \leq n_0$.
\end{proof}

We conclude with a corollary of this result.  Recall from \cite{jo-fpff} that $\indset{dnc} <_s \indset{dnc}_n$.  Hence, there are Medvedev reductions of $\indset{im}$ and $\indset{bi}$ to $\indset{dnc}_n$, for each $n \in \omega$.  (See Definition~\ref{dnc def} for the definition of $\indset{dnc}_n$.)

\begin{Corollary}
The index sets $\indset{dnc}_n [\indset{im}]$ and $\indset{dnc}_n [\indset{bi}]$ are both $\Pi^1_1$-complete.
\end{Corollary}

\begin{proof}
In light of Theorem~\ref{pi 2 version}, it will suffice to produce a $\Delta^0_2$ embedding, $f$, of $\omega^\omega$ into $\indset{dnc}_n$, such that the range of $f$ the intersection of $\indset{dnc}_n$ with a $\Pi^0_2$ class.  Note that such an embedding cannot have a closed range, otherwise it would be a homeomorphism between $\omega^\omega$ and a compact space.  

Let $\varphi : \omega^2 \rightarrow n$ be a universal partial recursive function and, for each $s \in \omega$, let $x_{0,s} < x_{1,s} < \ldots$ enumerate those $x$ such that $\varphi_x (x)$ has not yet converged in $s$ stages.  Each sequence $(x_{i,s})_{s \in \omega}$ is eventually constant, say with limit $x_i$.  Observe that $x_0 < x_1 < \ldots$ enumerate those $x$ such that $\varphi_x (x)$ diverges.

Fix a $\indset{dnc}_n$ function, $H$.  For each $s \in \omega$ and $\alpha\in \omega^{<\omega}$, define $f_s (\alpha) \in n^{<\omega}$ of length 
\[
x_{(\alpha(0) + \ldots + \alpha(|\alpha| - 1)  + |\alpha|-1) , s} + 1
\]
such that
\[
f_s (\alpha) (j) = \begin{cases}
1 &\mbox{if } (\exists i,p) (j = x_{i,s} \wedge i = \alpha(0) + \ldots + \alpha(p) + p),\\
0 &\mbox{if } (\exists i) (j = x_{i,s}) \mbox{, but there is no $p$ as above},\\
H(j) &\mbox{otherwise}.
\end{cases}
\]
Let $f : \omega^\omega \rightarrow \indset{dnc}_n$ be the $\Delta^0_2$ embedding induced by the $f_s$.  It follows that, for each $Y \in \omega^\omega$ and each $j \in \omega$
\[
f(Y)(j) = \begin{cases}
1 &\mbox{if } (\exists i,p) (j = x_i \wedge i = Y(0) + \ldots + Y(p) + p),\\
0 &\mbox{if } (\exists i) (j = x_i) \mbox{, but there is no $p$ as above},\\
H(j) &\mbox{otherwise}.
\end{cases}
\]

Note that, if $G \in \indset{dnc}_n$, then $G(x) = H(x)$, for each $x \notin \{ x_0 , x_1 , \ldots \}$.  Hence,
\[
\mathrm{range} (f) = \indset{dnc}_n \cap \{ Y \in n^\omega : (\forall i) (Y(x_i) \in \{0,1\}) \wedge (\exists^\infty i) (Y(x_i) = 1) \}.
\]
Observe that, since the sequence $(x_i)_{i \in \omega}$ is limit computable, it follows that the latter set in the intersection above is $\Pi^0_2$.  In other words, $\mathrm{range} (f)$ is the intersection of $\indset{dnc}_n$ with a $\Pi^0_2$ class.  This completes the proof.
\end{proof}

\end{document}